\documentclass[11pt]{article}
\usepackage{amsmath}
\usepackage{amssymb}
\usepackage{amsthm}
\usepackage[usenames]{color}
\usepackage{amscd}
\usepackage{dsfont}
\usepackage{indentfirst}

\usepackage[colorlinks=true,linkcolor=blue,filecolor=red,
citecolor=webgreen]{hyperref}
\definecolor{webgreen}{rgb}{0,.5,0}

\def\C{{\mathds{C}}}

\def\N{{\mathds{N}}}

\def\1{{\bf 1}}

\newtheorem{theorem}{Theorem}[section]

\newtheorem{cor}{Corollary}

\newtheorem{prop}[theorem]{Proposition}

\begin{document}

\title{\bf On multivariable averages of divisor functions}
\author{L\'aszl\'o T\'oth and Wenguang Zhai}
\date{}
\maketitle

\centerline{Journal of Number Theory {\bf 192} (2018), 251--269}

\begin{abstract} We deduce asymptotic formulas for the sums
$\sum_{n_1,\ldots,n_r\le x} f(n_1\cdots n_r)$ and \\
$\sum_{n_1,\ldots,n_r\le x} f([n_1\cdots n_r])$, where $r\ge 2$ is a
fixed integer, $[n_1,\ldots,n_r]$ stands for the least common
multiple of the integers $n_1,\ldots,n_r$ and $f$ is one of the
divisor functions $\tau_{1,k}(n)$ ($k\ge 1$), $\tau^{(e)}(n)$ and $\tau^*(n)$.
Our formulas refine and generalize a result of Lelechenko (2014).
A new generalization of the Busche-Ramanujan
identity is also pointed out.
\end{abstract}



\section{Introduction}

Let $\tau_{1,k}(n)=\sum_{ab^k=n} 1$, where $k\in \N:=\{1,2,\ldots \}$ is a fixed integer. For $k=1$ this is the
divisor function $\tau(n)$. It is well known that
\begin{equation} \label{Dir_divisor}
\sum_{n\le x} \tau(n)= x\log x +(2\gamma-1)x + O(x^{\theta+\varepsilon}),
\end{equation}
for any $\varepsilon>0$, where $\gamma$ is Euler's constant and $1/4\le  \theta< 1/3$. The best
result up to date, namely $\theta \le 517/1648=0.3137\ldots$ is a due to Bourgain and Watt \cite{BouWat2017}.
Furthermore, for $k\ge 2$,
\begin{equation} \label{tau_1dim}
\sum_{n\le x} \tau_{1,k}(n)= \zeta(k) x + \zeta(1/k) x^{1/k} + O(x^{\theta_k+\varepsilon}),
\end{equation}
where $1/(2(k+1))\le  \theta_k\le 1/(k+2)$, which can be improved.
See, e.g., the book by Kr\"{a}tzel \cite[Ch.\ 5]{Kra1988}. Note that
$\theta_2 \le \frac{1057}{4785}= 0.2208\ldots$, which is a result of Graham and
Kolesnik \cite{GraKol1988}.

The exponential divisor function $\tau^{(e)}$ is multiplicative and defined by
$\tau^{(e)}(p^{\nu})=\tau(\nu)$ for every prime power $p^{\nu}$ ($\nu \in \N$).
It is known that
\begin{equation*}
\sum_{n\le x} \tau^{(e)}(n) = c_1 x + c_2 x^{1/2}+ O(x^{\lambda+\varepsilon}),
\end{equation*}
for every $\varepsilon >0$, where $c_1, c_2$ are computable constants and
$\lambda=\frac{1057}{4785}= 0.2208\ldots$, the error term being strongly related to the divisor function
$\tau_{1,2}$. See Wu \cite{Wu1995}.

Lelechenko \cite{Lel2014} proved, using a multidimensional Perron formula and the complex integration method that
\begin{equation} \label{tau_2dim}
\sum_{m,n\le x} \tau_{1,2}(mn)= A_2 x^2 + B_2 x^{3/2} + O(x^{10/7+\varepsilon}),
\end{equation}
where $A_2,B_2$ are constants and $10/7=1.4285\ldots$. He noted that in the case $k\ge 3$ the same method does not furnish the expected
asymptotic formula
\begin{equation} \label{tau_1k}
\sum_{m,n\le x} \tau_{1,k}(mn)= A_k x^2 + B_k x^{1+1/k} + O(x^{\alpha_k+\varepsilon}),
\end{equation}
since the obtained error term is larger than $x^{4/3}$, even under the Riemann hypothesis, and
absorbs the term $x^{1+1/k}$. It is also noted in paper \cite{Lel2014} that formula \eqref{tau_2dim} remains valid for the function
$\tau^{(e)}$ instead of $\tau_{1,2}$, due to the fact that $\tau^{(e)}(p^\nu)=\tau_{1,2}(p^\nu)$ for $\nu \in \{1,2,3,4\}$.

Another similar divisor function is $\tau^*(n)=2^{\omega(n)}$, representing the number of unitary divisors of $n$, which equals the number of
squarefree divisors of $n$. One has
\begin{equation*}
\sum_{n\le x} \tau^*(n) = \frac{6}{\pi^2}x \left(\log x + 2\gamma-1-\frac{2\zeta'(2)}{\zeta(2)}\right) + O(R(x)),
\end{equation*}
where $R(x)\ll x^{1/2}\exp(-c_0(\log x)^{3/5}(\log \log x)^{-1/5})$
with $c_0$ a positive constant. See \cite{SS1970}.

It is the goal of the present paper to improve the error term of
\eqref{tau_2dim} and to deduce formula \eqref{tau_1k} with a sharp error term.
More generally, we derive asymptotic formulas for the sums
$\sum_{n_1,\ldots,n_r\le x} \tau_{1,k}(n_1\cdots n_r)$ and
$\sum_{n_1,\ldots,n_r\le x} \tau_{1,k}([n_1,\ldots, n_r])$, where
$k\ge 1$ and $r\ge 2$ are fixed integers and $[n_1,\ldots n_r]$
stands for the least common multiple of $n_1,\ldots,n_r$.

Furthermore, we deduce similar asymptotic formulas concerning the divisor functions $\tau^{(e)}$ and $\tau^*$.
Our approach is based on the study of multiple Dirichlet series and the convolution method. A new
generalization of the Busche-Ramanujan identity is also pointed out.

We remark that asymptotic formulas for sums $\sum_{n_1,\ldots,n_r\le
x} f([n_1,\ldots,n_r])$, where $f$ belongs to a large class of
multiplicative functions, including $\sigma_t(n)=\sum_{d\mid n} d^t$
and $\phi_t(n)=\sum_{d\mid n} d^t\mu(n/d)$ with $t\ge 1/2$ were
established by Hilberdink and the first author \cite{HilTot2016}. However, those results
can not be applied for the divisor functions investigated in the present paper.

\section{Multiple Dirichlet series}

The Dirichlet series of a function $f:\N^r\to \C$ is given by
\begin{equation*}
D(f;s_1,\ldots,s_r)= \sum_{n_1,\ldots,n_r=1}^{\infty}
\frac{f(n_1,\ldots,n_r)}{n_1^{s_1}\cdots n_r^{s_r}}.
\end{equation*}

Similar to the one variable case, if $D(f;s_1,\ldots,s_r)$ is
absolutely convergent for $(s_1,\ldots,s_r)\in \C^r$, then it is
absolutely convergent for every $(z_1,\ldots,z_r)\in \C^r$ with $\Re
z_j \ge \Re s_j$ ($1\le j\le r$).

The Dirichlet convolution of the functions $f,g:\N^r\to \C$ is defined by
\begin{equation*}
(f*g)(n_1,\ldots,n_r)= \sum_{d_1\mid n_1, \ldots, d_r\mid n_r}
f(d_1,\ldots,d_r) g(n_1/d_r, \ldots, n_r/d_r).
\end{equation*}

If $D(f;s_1,\ldots,s_r)$ and $D(g;s_1,\ldots,s_r)$ are
absolutely convergent, then
$D(f*g;z_1,\ldots,z_r)$ is also absolutely convergent and
\begin{equation*}
D(f*g;s_1,\ldots,s_r) = D(f;s_1,\ldots,s_r) D(g;s_1,\ldots,s_r).
\end{equation*}

We recall that a nonzero arithmetic function of $r$ variables $f:\N^r\to \C$ is said to be multiplicative if
\begin{equation*}
f(m_1n_1,\ldots,m_rn_r)= f(m_1,\ldots,m_r) f(n_1,\ldots,n_r)
\end{equation*}
holds for any $m_1,\ldots,m_r,n_1,\ldots,n_r\in \N$ such that $(m_1\cdots m_r,n_1\cdots n_r)=1$.
If $f$ is multiplicative, then it is determined by the values
$f(p^{\nu_1},\ldots,p^{\nu_r})$, where $p$ is prime and
$\nu_1,\ldots,\nu_r\in \N \cup \{0\}$. More exactly, $f(1,\ldots,1)=1$ and
for any $n_1,\ldots,n_r\in \N$,
\begin{equation*}
f(n_1,\ldots,n_r)= \prod_p f(p^{\nu_p(n_1)}, \ldots,p^{\nu_p(n_r)}),
\end{equation*}
where we use the notation $n=\prod_p p^{\nu_p(n)}$ for the prime power factorization of $n\in \N$, the product being over the
primes $p$ and all but a finite number of the exponents $\nu_p(n)$ are zero. If $r=1$, i.e., in the case of functions of a
single variable we reobtain the familiar notion of multiplicativity.

If $f$ is multiplicative, then its Dirichlet series can be expanded into a (formal)
Euler product, that is,
\begin{equation} \label{Euler_product}
D(f;s_1,\ldots,s_r)=  \prod_p \sum_{\nu_1,\ldots,\nu_r=0}^{\infty}
\frac{f(p^{\nu_1},\ldots, p^{\nu_r})}{p^{\nu_1s_1+\cdots +\nu_r
s_r}},
\end{equation}
the product being over the primes $p$. More exactly, for $f$ multiplicative, the series
$D(f;s_1,\ldots,s_r)$ is absolutely convergent if and only if
\begin{equation*}
\sum_p \sum_{\substack{\nu_1,\ldots,\nu_r=0\\ \nu_1+\cdots +\nu_r \ge 1}}^{\infty}
\frac{|f(p^{\nu_1},\ldots, p^{\nu_r})|}{p^{\nu_1 \Re s_1+\cdots +\nu_r
\Re s_r}} < \infty
\end{equation*}
and in this case equality \eqref{Euler_product} holds. See Delange \cite[Lemma 2]{Del1969} for a proof of this property in the case
of two variables. Also see the survey paper \cite{Tot2014} on general accounts concerning
multiplicative functions of several variables.

Now we consider the function $\tau_{1,k}$.

\begin{prop} \label{Prop_Dir_ser_tau_prod} Let $r\ge 2$, $k\ge 1$ and let $s_1,\ldots,s_r\in \C$ with $\Re s_j >1$ \textup{($1\le j\le r$)}. Then
\begin{equation*}
\sum_{n_1,\ldots, n_r=1}^{\infty} \frac{\tau_{1,k}(n_1\cdots n_r)}{n_1^{s_1}\cdots n_r^{s_r}} = \zeta(s_1)\zeta(ks_1)\cdots
\zeta(s_r)\zeta(ks_r) F_{r,k}(s_1,\ldots,s_r),
\end{equation*}
where
\begin{equation*}
F_{r,k}(s_1,\ldots,s_r) = \sum_{n_1,\ldots, n_r=1}^{\infty} \frac{f_{r,k}(n_1,\ldots, n_r)}{n_1^{s_1}\cdots n_r^{s_r}}
\end{equation*}
is absolutely convergent provided that $\Re s_j >0$ \textup{($1\le j\le r$)} and $\Re(s_j+s_{\ell})>1$ \textup{($1\le j<\ell \le r$)}.
\end{prop}

We remark that the following exact identity, valid for $k=1$ was
proved by the first author \cite[Eq.\ (4.2)]{Tot2013}:

\begin{equation*}
\sum_{n_1,\ldots,n_r=1}^{\infty} \frac{\tau(n_1\cdots
n_r)}{n_1^{s_1}\cdots n_r^{s_r}}
\end{equation*}
\begin{equation*}
= \zeta^2(s_1)\cdots \zeta^2(s_r) \prod_p \left(1+ \sum_{j=2}^r
(-1)^{j-1} (j-1) \sum_{1\le i_1< \cdots < i_j\le r}
\frac1{p^{s_{i_1}+\cdots +s_{i_j}}}\right),
\end{equation*}
where the infinite product is absolutely convergent if $\Re(s_{i_1}+\cdots +s_{i_j})>1$
($1\le i_1< \cdots <i_j \le r$ with $2\le j\le r$).

In the case of two variables, i.e., $r=2$ and $k\ge 1$ arbitrary we
have the following explicit formula:

\begin{prop} \label{Prop_Dir_series_r_2} Let $k\ge 1$ and let $s_1,s_2\in \C$ with $\Re s_1, \Re s_2 >1$, Then
\begin{equation*}
\sum_{n_1,n_2=1}^{\infty} \frac{\tau_{1,k}(n_1n_2)}{n_1^{s_1}n_2^{s_2}} = \zeta(s_1)\zeta(ks_1)\zeta(s_2)\zeta(ks_2) F_k(s_1,s_2)
\end{equation*}
where
\begin{equation*}
F_k(s_1,s_2) = \prod_p \left(1+ \sum_{j=1}^{k-1} \frac1{p^{js_1+(k-j)s_2}} - \sum_{j=1}^k \frac1{p^{js_1+(k+1-j)s_2}} \right)
\end{equation*}
is absolutely convergent if $\Re(js_1+(k-j)s_2)>1$ \textup{($1\le j\le k-1$)} and $\Re(js_1+(k+1-j)s_2)>1$ \textup{($1\le j\le k$)}.
\end{prop}

\begin{cor} \label{Cor_BR} Let $k\ge 1$. For every $n_1,n_2\in \N$,
\begin{equation} \label{BRk}
\tau_{1,k}(n_1n_2) = \sum_{\substack{d_1\mid n_1 \\ d_2\mid n_2}} f_k(d_1,d_2) \tau_{1,k}(n_1/d_1)\tau_{1,k}(n_2/d_2),
\end{equation}
where the function $f_k(n_1,n_2)$ is multiplicative and for prime powers $p^{\nu_1}$, $p^{\nu_2}$,
\begin{equation*}
f_k(p^{\nu_1},p^{\nu_2}) = \begin{cases}
1, & \text{ if $\nu_1=\nu_2=0$ or $\nu_1,\nu_2\ge 1$ and $\nu_1+\nu_2=k$},\\
-1, & \text{ if $\nu_1,\nu_2\ge 1$ and $\nu_1+\nu_2=k+1$},\\
0, & \text{ otherwise}.
\end{cases}
\end{equation*}
\end{cor}

Note that \eqref{BRk} is a generalization of the Busche-Ramanujan formula
\begin{equation*}
\tau(n_1n_2) = \sum_{d\mid \gcd(n_1,n_2)} \mu(d) \tau(n_1/d)\tau(n_2/d),
\end{equation*}
valid for every $n_1,n_2\in \N$, which is recovered in the case $k=1$.

\begin{prop} \label{Prop_Dir_ser_tau_lcm} Let $r\ge 2$, $k\ge 1$ and let $s_1,\ldots,s_r\in \C$ with $\Re s_j >1$ \textup{($1\le j\le r$)}. Then
\begin{equation*}
\sum_{n_1,\ldots, n_r=1}^{\infty} \frac{\tau_{1,k}([n_1,\ldots, n_r])}{n_1^{s_1}\cdots n_r^{s_r}} = \zeta(s_1)\zeta(ks_1)\cdots
\zeta(s_r)\zeta(ks_r) G_{r,k}(s_1,\ldots,s_r),
\end{equation*}
where
\begin{equation*}
G_{r,k}(s_1,\ldots,s_r) = \sum_{n_1,\ldots, n_r=1}^{\infty} \frac{g_{r,k}(n_1,\ldots, n_r)}{n_1^{s_1}\cdots n_r^{s_r}}
\end{equation*}
is absolutely convergent provided that $\Re s_j >0$ \textup{($1\le j\le r$)} and $\Re(s_j+s_{\ell})>1$ \textup{($1\le j<\ell \le r$)}.
\end{prop}

Concerning the exponential divisor function $\tau^{(e)}$ we have

\begin{prop} \label{Prop_Dir_ser_tau_exp_prod_lcm} Let $r\ge 2$ and let $s_1,\ldots,s_r\in \C$ with
$\Re s_j >1$ \textup{($1\le j\le r$)}. Then
\begin{equation*}
\sum_{n_1,\ldots, n_r=1}^{\infty} \frac{\tau^{(e)}(n_1\cdots n_r)}{n_1^{s_1}\cdots n_r^{s_r}} = \zeta(s_1)\zeta(2s_1) \cdots
\zeta(s_r)\zeta(2s_r) T_r(s_1,\ldots,s_r),
\end{equation*}
and
\begin{equation*}
\sum_{n_1,\ldots, n_r=1}^{\infty} \frac{\tau^{(e)}([n_1,\ldots
n_r])}{n_1^{s_1}\cdots n_r^{s_r}} = \zeta(s_1)\zeta(2s_1) \cdots
\zeta(s_r)\zeta(2s_r) V_r(s_1,\ldots,s_r),
\end{equation*}
where the Dirichlet series
\begin{equation*}
T_r(s_1,\ldots,s_r) = \sum_{n_1,\ldots, n_r=1}^{\infty} \frac{t(n_1,\ldots, n_r)}{n_1^{s_1}\cdots n_r^{s_r}}
\end{equation*}
and
\begin{equation*}
V_r(s_1,\ldots,s_r) = \sum_{n_1,\ldots, n_r=1}^{\infty}
\frac{v(n_1,\ldots, n_r)}{n_1^{s_1}\cdots n_r^{s_r}}
\end{equation*}
are both absolutely convergent if $\Re s_j >1/5$ \textup{($1\le j\le r$)}
and $\Re (s_j+s_{\ell}) >1$ {\rm ($1\le j<\ell \le r$)}.
\end{prop}

Now we move to the function $\tau^*$. Note that $\tau^*(n_1\cdots n_r)=\tau^*([n_1,\ldots,n_r])$ for every $n_1,\ldots,n_r\in \N$.
We have the next result:

\begin{prop} \label{Prop_Dir_ser_tau_star_prod} Let $r\ge 2$ and let $s_1,\ldots,s_r\in \C$ with $\Re s_j >1$ \textup{($1\le j\le r$)}. Then
\begin{equation*}
\sum_{n_1,\ldots, n_r=1}^{\infty} \frac{\tau^*(n_1\cdots n_r)}{n_1^{s_1}\cdots n_r^{s_r}} = \zeta^2(s_1)\cdots
\zeta^2(s_r) H_r(s_1,\ldots,s_r),
\end{equation*}
where
\begin{equation*}
H_r(s_1,\ldots,s_r) = \prod_p \left(1-\frac1{p^{s_1}} \right)\cdots \left(1-\frac1{p^{s_r}} \right) \left(2- \left(1-\frac1{p^{s_1}}
\right)\cdots \left(1-\frac1{p^{s_r}} \right) \right)
\end{equation*}
is absolutely convergent if $\Re s_j >1/2$ \textup{($1\le j\le r$)}.
\end{prop}

\section{Asymptotic formulas}

We prove the following results.

\begin{theorem}  \label{Th_asympt_tau_1k} If $k,r\ge 2$, then
\begin{equation} \label{asympt_tau_1k}
\sum_{n_1,\ldots,n_r\le x} \tau_{1,k}(n_1\cdots n_r)= A_{r,k} x^r + B_{r,k} x^{r-1+1/k} +
O(x^{r-1+\theta_k+\varepsilon}),
\end{equation}
\begin{equation} \label{asympt_tau_ik_lcm}
\sum_{n_1,\ldots,n_r\le x} \tau_{1,k}([n_1, \ldots, n_r]) = C_{r,k} x^r + D_{r,k} x^{r-1+1/k} +
O(x^{r-1+\theta_k+\varepsilon}),
\end{equation}
for every $\varepsilon >0$, where   $\theta_k$ is the exponent in the
error term of formula \eqref{tau_1dim} and $A_{r,k}$, $B_{r,k}$, $C_{r,k}$, $D_{r,k}$ are computable constants. Here
\begin{equation*}
A_{r,k}:= \prod_p \left(1-\frac1{p}\right)^r \sum_{\nu_1,\ldots,\nu_r=0}^{\infty}
\frac{\lfloor (\nu_1 +\cdots +\nu_r)/k \rfloor +1}{p^{\nu_1+\cdots+ \nu_r}},
\end{equation*}
\begin{equation*}
C_{r,k}:= \prod_p \left(1-\frac1{p}\right)^r \sum_{\nu_1,\ldots,\nu_r=0}^{\infty}
\frac{\lfloor \max(\nu_1,\ldots,\nu_r)/k \rfloor +1}{p^{\nu_1+\cdots+ \nu_r}}.
\end{equation*}
\end{theorem}

It follows from Proposition \ref{Prop_Dir_series_r_2} that for $r=2$
and $k\ge 2$,
\begin{equation*}
A_{2,k} = \zeta^2(k) F_k(1,1)= \zeta^2(k) \prod_p
\left(1+\frac{k-1}{p^k}-\frac{k}{p^{k+1}} \right).
\end{equation*}

\begin{theorem} \label{Th_asympt_tau_exp} If $r\ge 2$, then
\begin{equation*}
\sum_{n_1,\ldots,n_r\le x} \tau^{(e)}(n_1\cdots n_r)= K_r x^r + L_r
x^{r-1/2} + O(x^{r-1+\theta_2+ \varepsilon}),
\end{equation*}
\begin{equation*}
\sum_{n_1,\ldots,n_r\le x} \tau^{(e)}([n_1,\ldots, n_r])= K'_r x^r +
L'_r x^{r-1/2} + O(x^{r-1+\theta_2+ \varepsilon}),
\end{equation*}
for every $\varepsilon >0$, where $\theta_2\le 0.2208 \ldots$ is
defined by \eqref{tau_1dim} and $K_r$, $L_r$, $K'_r$, $L'_r$ are computable constants. Here
\begin{equation*}
K_r=\prod_p \left(1-\frac1{p} \right)^r \left( 1+
\sum_{\substack{\nu_1,\ldots,\nu_r=0\\ \nu_1+\cdots +\nu_r \ge 1
}}^{\infty} \frac{\tau(\nu_1+\cdots +\nu_r)}{p^{\nu_1+\cdots+
\nu_r}}\right),
\end{equation*}
\begin{equation*}
K'_r= \prod_p \left(1-\frac1{p} \right)^r \left( 1+
\sum_{\substack{\nu_1,\ldots,\nu_r=0\\ \nu_1+\cdots +\nu_r\ge 1
}}^{\infty} \frac{\tau(\max(\nu_1,\ldots,\nu_r))}{p^{\nu_1+\cdots+
\nu_r}}\right).
\end{equation*}
\end{theorem}

Our multivariable asymptotic formulas regarding the divisor functions $\tau$ and $\tau^*$ are special cases of
the following general convolution result.

\begin{theorem} \label{Th_gen_convo} Let $r\ge 2$ and let
$h:\N^r\to \C$, $g:\N^r\to \C$, $f_j:\N\to \C$ \textup{($1\le j\le r$)} be arithmetic functions such that
\begin{equation*}
h(n_1,\ldots,n_r) = \sum_{d_1m_1=n_1,\ldots, d_rm_r=n_r} g(d_1,\ldots, d_r) f_1(m_1)\cdots f_r(m_r)
\end{equation*}
for every $n_1,\ldots,n_r\in \N$. Assume that

(i) there exist constants $0<b_j<a_j$ \textup{($1\le j\le r$)} such that
\begin{equation*}
F_j(x):= \sum_{n\le x} f_j(n)= x^{a_j}P_j(\log x)+ O(x^{b_j}) \quad (1\le j\le r),
\end{equation*}
where $P_j(u)$ are polynomials in $u$ of degrees $\delta_j$, with leading coefficients $K_j$ \textup{($1\le j\le r$)},

(ii) the Dirichlet series
\begin{equation*}
G(s_1,\ldots,s_r):= \sum_{n_1,\ldots, n_r=1}^{\infty} \frac{g(n_1,\ldots,n_r)}{n_1^{s_1}\cdots n_r^{s_r}}
\end{equation*}
is absolutely convergent for $(s_1,\ldots, s_r)=(a_1-\varepsilon,\ldots,a_{j-1}-\varepsilon, b_j-\varepsilon, a_{j+1}-\varepsilon,
\ldots, a_r-\varepsilon)$ for sufficiently small $\varepsilon>0$ and $1\le j\le r$.

Then the asymptotic formula
\begin{equation*}
\sum_{n_1,\ldots, n_r\le x} h(n_1,\ldots,n_r) = x^{a_1+\cdots+a_r} Q(\log x) + O(x^{a_1+\cdots+a_r -\Delta}(\log x)^{\delta_1+\cdots +\delta_r})
\end{equation*}
holds, where $Q(u)$ is a polynomial in $u$ of degree $\delta_1+\cdots +\delta_r$, with leading coefficient
$K_1\cdots K_r G(a_1,\ldots,a_r)$ and $\Delta =\min_{1\le j\le r} (a_j-b_j)$.
\end{theorem}

\begin{theorem} \label{Th_tau_r} If $r\ge 2$, then
\begin{equation*} \label{form_tau_r}
\sum_{n_1,\ldots,n_r \le x} \tau(n_1\cdots n_r) = x^r P_r(\log x)+ O(x^{r-1+\theta+\varepsilon}),
\end{equation*}
\begin{equation*} \label{form_tau_lcm_r}
\sum_{n_1,\ldots,n_r \le x} \tau([n_1,\ldots, n_r]) = x^r Q_r(\log x)+ O(x^{r-1+\theta+\varepsilon}),
\end{equation*}
for every $\varepsilon >0$, where $\theta$ is the exponent in \eqref{Dir_divisor}, $P_r(t)$ and $Q_r(t)$ are polynomials in $t$ of
degree $r$ having the leading coefficients
\begin{equation} \label{K_P_r}
K_{P,r}:=\prod_p \left(1-\frac1{p} \right)^{r-1}
\left(1+\frac{r-1}{p}\right)
\end{equation}
and
\begin{equation*} \label{K_Q_r}
K_{Q,r}:= \prod_p \left(1-\frac1{p}\right)^{2r} \sum_{\nu_1,\ldots,\nu_r=0}^{\infty} \frac{\max(\nu_1,\ldots,\nu_r)+1}{p^{\nu_1+\cdots+ \nu_r}},
\end{equation*}
respectively.
\end{theorem}

Note that the constant $K_{P,r}$ defined by \eqref{K_P_r} equals the
asymptotic density of the set of $r$-tuples of positive integers
with pairwise relatively prime components. See
\cite{Tot2014,Tot2016}.

Furthermore, in the case $r=2$,
\begin{equation*} \label{K_r}
K_{Q,2}= \zeta(2) K_{P,3} =  \zeta(2) \prod_p \left(1-\frac1{p} \right)^2
\left(1+\frac{2}{p}\right).
\end{equation*}

\begin{theorem}  \label{Th_tau_star_r} If $r\ge 2$, then
\begin{equation*}
\sum_{n_1,\ldots,n_r\le x} \tau^*(n_1\cdots n_r)= x^r P^*_r(\log x)+ O(x^{r-1/2+\varepsilon}),
\end{equation*}
where $P^*_r(t)$ is a polynomials in $t$ of degree $r$ having the leading coefficient
\begin{equation*}
K_{P^*,r}:=\prod_p \left(1-\frac1{p} \right)^r \left(2- \left(1-\frac1{p} \right)^r \right).
\end{equation*}
\end{theorem}

It is easier to derive asymptotic formulas for similar sums involving the greatest common divisor
$(n_1,\ldots,n_r)$. For example, we have the following result.

\begin{theorem} \label{Th_tau_1_k_gcd_r} Let $r\ge 2$ and $k\ge 1$. Then
\begin{equation} \label{tau_1_k_gcd_r}
\sum_{n_1,\ldots,n_r\le x} \tau_{1,k}((n_1,\ldots,n_r))= \zeta(kr)x^r + O(R_{r,k}(x)),
\end{equation}
where $R_{r,k}(x)=x^{r-1}$ \textup{($rk>2$)} and $R_{2,1}(x)= x\log x$.
\end{theorem}

\section{Proofs}

\begin{proof}[Proof of Proposition {\rm \ref{Prop_Dir_ser_tau_prod}}]
The function $n\mapsto \tau_{1,k}(n)$ is multiplicative and $\tau_{1,k}(p^\nu)=\lfloor \nu/k \rfloor +1$ for every prime power
$p^{\nu}$ ($\nu \ge 0$). The function $(n_1,\ldots,n_r)\mapsto \tau_{1,k}(n_1\cdots n_r)$ is multiplicative,
viewed as a function of $r$ variables. Therefore, its multiple Dirichlet series can be expanded into an Euler product.
We obtain
\begin{equation*}
\sum_{n_1,\ldots, n_r=1}^{\infty} \frac{\tau_{1,k}(n_1\cdots n_r)}{n_1^{s_1}\cdots n_r^{s_r}} = \prod_p
\sum_{\nu_1,\ldots,\nu_r=0}^{\infty} \frac{\tau_{1,k}(p^{\nu_1+\cdots +\nu_r})}{p^{\nu_1 s_1+\cdots +\nu_r s_r}}
\end{equation*}
\begin{equation*}
= \prod_p \sum_{\nu_1,\ldots,\nu_r=0}^{\infty} \frac{\lfloor (\nu_1 +\cdots +\nu_r)/k \rfloor +1}{p^{\nu_1 s_1+\cdots +\nu_r s_r}}
\end{equation*}
\begin{equation*}
= \zeta(s_1)\zeta(ks_1)\cdots \zeta(s_r)\zeta(ks_r) F_{r,k}(s_1,\ldots,s_r),
\end{equation*}
where
\begin{equation*}
F_{r,k}(s_1,\ldots,s_r)
\end{equation*}
\begin{equation*}
= \prod_p \left(1-\frac1{p^{s_1}}\right) \left(1-\frac1{p^{ks_1}}\right)\cdots \left(1-\frac1{p^{s_r}}\right)
\left(1-\frac1{p^{ks_r}}\right)
\end{equation*}
\begin{equation*}
\times \sum_{\nu_1,\ldots,\nu_r=0}^{\infty} \frac{\lfloor
(\nu_1 +\cdots +\nu_r)/k \rfloor +1}{p^{\nu_1 s_1+\cdots +\nu_r s_r}}
\end{equation*}
\begin{equation*}
= \prod_p \left(1-\frac1{p^{s_1}}- \frac1{p^{ks_1}} +\frac1{p^{(k+1)s_1}}\right) \cdots
\left(1-\frac1{p^{s_r}}- \frac1{p^{ks_r}}+\frac1{p^{(k+1)s_r}} \right)
\end{equation*}
\begin{equation*}
\times \sum_{\nu_1,\ldots,\nu_r=0}^{\infty}
\frac{\lfloor (\nu_1 +\cdots +\nu_r)/k \rfloor +1}{p^{\nu_1 s_1+\cdots +\nu_r s_r}}
\end{equation*}
\begin{equation*}
= \prod_p \left(1+ \sum_{\substack{\nu_1,\ldots,\nu_r=0\\ \# A(\nu_1,\ldots,\nu_r)\ge 2 }}^{\infty}
\frac{c_{\nu_1,\ldots,\nu_r}}{p^{\nu_1 s_1+\cdots +\nu_r s_r}} \right)
\end{equation*}
with some coefficients $c_{\nu_1,\ldots,\nu_r}$, where $A(\nu_1,\ldots,\nu_r):=\{j: 1\le j\le r, \nu_j\ne 0 \}$. Here the coefficient
$c$ of $1/p^{\ell s_j}$ (the case $\nu_t=0$ for all $t\ne j$ and $\nu_j=\ell$) is zero for every $1\le j\le r$ and $\ell \ge 1$. Indeed,
\begin{equation*}
c=\begin{cases}
\left(\lfloor \frac{\ell}{k} \rfloor + 1\right) - \left(\lfloor \frac{\ell-1}{k} \rfloor +1\right)=0, & \text{if $1\le \ell \le k-1$}, \\
\left(\lfloor \frac{k}{k} \rfloor + 1\right) - \left( \lfloor \frac{k-1}{k} \rfloor +1\right) -1 =0, &\text{if $\ell=k$}, \\
\left(\lfloor \frac{\ell}{k} \rfloor + 1\right) - \left(\lfloor \frac{\ell-1}{k}\rfloor +1\right)- \left(\lfloor \frac{\ell-k}{k} \rfloor +1 \right)
\\ + \left(\lfloor \frac{\ell-k-1}{k} \rfloor +1\right) =0, & \text{if $\ell \ge k+1$}.
\end{cases}
\end{equation*}

Hence a sufficient condition of absolute convergence of $F_{r,k}(s_1,\ldots,s_r)$ is that $\Re s_j >0$
($1\le j\le r$) and $\Re(s_j+s_{\ell})>1$ ($1\le j<\ell \le r$). Note that another sufficient condition for absolute convergence
is $\Re s_j >1/2$ ($1\le j\le r$), which can not be used in the proof of the corresponding asymptotic formula.
\end{proof}

\begin{proof}[Proof of Proposition {\rm \ref{Prop_Dir_series_r_2}}]
In the case $r=2$ we have
\begin{equation*}
D_k(s_1,s_2):= \sum_{n_1,n_2=1}^{\infty} \frac{\tau_{1,k}(n_1n_2)}{n_1^{s_1}n_2^{s_2}}
=  \prod_p \sum_{\nu_1,\nu_2=0}^{\infty} \frac{\lfloor (\nu_1+\nu_2)/k \rfloor +1}{p^{\nu_1 s_1+\nu_2 s_2}}
\end{equation*}
\begin{equation*}
=  \prod_p \sum_{j=0}^{\infty} \sum_{\substack{\nu_1,\nu_2=0\\jk\le \nu_1+\nu_2 \le (j+1)k-1}}^{\infty} \frac{j +1}{p^{\nu_1 s_1+\nu_2 s_2}}
=  \prod_p \sum_{j=0}^{\infty} (j+1) \sum_{\ell=jk}^{(j+1)k-1}
\frac1{p^{\ell s_2}} \sum_{\nu_1=0}^{\ell}
\frac1{p^{\nu_1(s_1-s_2)}}.
\end{equation*}

Let $x=p^{-s_1}$, $y=p^{-s_2}$. We deduce that $D_k(s_1,s_2)=  \prod_p
S_k(x,y)$, where
\begin{equation*}
S_k(x,y)=  \sum_{j=0}^{\infty} (j+1) \sum_{\ell=jk}^{(j+1)k-1} y^\ell \sum_{a=0}^{\ell} \left(\frac{x}{y}\right)^a
\end{equation*}
\begin{equation*}
= \sum_{j=0}^{\infty} (j+1) \sum_{\ell=jk}^{(j+1)k-1} y^\ell \left(1-\left(\frac{x}{y}\right)^{\ell+1}\right)\left(1-\frac{x}{y} \right)^{-1}
\end{equation*}
\begin{equation*}
=  \left(1-\frac{x}{y} \right)^{-1} \sum_{j=0}^{\infty} (j+1) \left( \sum_{\ell=jk}^{(j+1)k-1} y^\ell -\frac{x}{y}
\sum_{\ell=jk}^{(j+1)k-1} x^{\ell} \right)
\end{equation*}
\begin{equation*}
=  \left(1-\frac{x}{y} \right)^{-1} \left(\frac{1-y^k}{1-y}
\sum_{j=0}^{\infty} (j+1)y^{jk} -\frac{x}{y}\cdot \frac{1-x^k}{1-x}
\sum_{j=0}^{\infty} (j+1)x^{jk} \right)
\end{equation*}
\begin{equation*}
= \left(1-\frac{x}{y} \right)^{-1} \frac{(1-x)(1-x^k)-xy^{-1}(1-y)(1-y^k)}{(1-x)(1-y)(1-x^k)(1-y^k)} \\
\end{equation*}
\begin{equation*}
= \frac1{(1-x)(1-y)(1-x^k)(1-y^k)} \left(1+ \sum_{t=1}^{k-1} x^t
y^{k-t} - \sum_{t=1}^k x^t y^{k+1-t}\right).
\end{equation*}

This gives the result.
\end{proof}

\begin{proof}[Proof of Proposition {\rm \ref{Prop_Dir_ser_tau_lcm}}]
Similar to the proof of Proposition \ref{Prop_Dir_ser_tau_prod}. The function $(n_1,\ldots,n_r)\mapsto \tau_{1,k}([n_1,\ldots, n_r])$
is also multiplicative. Its multiple Dirichlet series can be written as
\begin{equation*}
\sum_{n_1,\ldots, n_r=1}^{\infty} \frac{\tau_{1,k}([n_1,\ldots, n_r])}{n_1^{s_1}\cdots n_r^{s_r}} = \prod_p
\sum_{\nu_1,\ldots,\nu_r=0}^{\infty} \frac{\tau_{1,k}(p^{\max(\nu_1,\ldots, \nu_r)})}{p^{\nu_1 s_1+\cdots +\nu_r s_r}}
\end{equation*}
\begin{equation*}
= \prod_p \sum_{\nu_1,\ldots,\nu_r=0}^{\infty} \frac{\lfloor \max(\nu_1, \ldots, \nu_r)/k \rfloor +1}{p^{\nu_1 s_1+\cdots +\nu_r s_r}}
\end{equation*}
\begin{equation*}
= \zeta(s_1)\zeta(ks_1)\cdots \zeta(s_r)\zeta(ks_r) G_{r,k}(s_1,\ldots,s_r),
\end{equation*}
where
\begin{equation*}
G_{r,k}(s_1,\ldots,s_r)
\end{equation*}
\begin{equation*}
= \prod_p \left(1-\frac1{p^{s_1}}- \frac1{p^{ks_1}} +\frac1{p^{(k+1)s_1}}\right) \cdots
\left(1-\frac1{p^{s_r}}- \frac1{p^{ks_r}}+\frac1{p^{(k+1)s_r}} \right)
\end{equation*}
\begin{equation*}
\times \sum_{\nu_1,\ldots,\nu_r=0}^{\infty}
\frac{\lfloor \max(\nu_1,\ldots,\nu_r)/k \rfloor +1}{p^{\nu_1 s_1+\cdots +\nu_r s_r}}
\end{equation*}
\begin{equation*}
= \prod_p \left(1+ \sum_{\substack{\nu_1,\ldots,\nu_r=0\\ \# A(\nu_1,\ldots,\nu_r)\ge 2 }}^{\infty}
\frac{c'_{\nu_1,\ldots,\nu_r}}{p^{\nu_1 s_1+\cdots +\nu_r s_r}} \right),
\end{equation*}
since here the coefficient $c'$ of $1/p^{\ell s_j}$ equals the coefficient $c$ of $1/p^{\ell s_j}$ in $F_{r,k}(s_1,\ldots,s_r)$, which vanishes,
as explained in the proof of Proposition \ref{Prop_Dir_ser_tau_prod}.
\end{proof}

\begin{proof}[Proof of Proposition {\rm \ref{Prop_Dir_ser_tau_exp_prod_lcm}}]
Similar to the proof of Proposition \ref{Prop_Dir_ser_tau_prod}.
Since $\tau^{(e)}(p^\nu)=\tau(\nu)=\tau_{1,2}(p^\nu)$ for $\nu \in
\{1,2,3,4\}$, in $T_r(s_1,\ldots,s_r)$ and $V_r(s_1,\ldots,s_r)$ the
coefficients of the terms $1/p^{\ell s_j}$ will be zero for every
$1\le j\le r$ and $\ell \in \{1,2,3,4\}$. However,
$\tau^{(e)}(p^5)=\tau(5)=2\ne 3=\tau_{1,2}(p^5)$, therefore the
coefficients of the terms $1/p^{5s_j}$ will not vanish (they are
$-1$ for every $1\le j\le r$). Hence a sufficient condition for absolute convergence
is that $\Re s_j >1/5$ {\rm ($1\le j\le r$)} and $\Re (s_j+s_{\ell}) >1$ {\rm ($1\le j< \ell \le r$)}.
\end{proof}

\begin{proof}[Proof of Proposition {\rm \ref{Prop_Dir_ser_tau_star_prod}}]
The function $(n_1,\ldots,n_r)\mapsto \tau^*(n_1\cdots n_r)$ is
multiplicative. Its multiple Dirichlet series can be written as
\begin{equation*}
\sum_{n_1,\ldots, n_r=1}^{\infty} \frac{\tau^*(n_1\cdots
n_r)}{n_1^{s_1}\cdots n_r^{s_r}} = \prod_p
\sum_{\nu_1,\ldots,\nu_r=0}^{\infty} \frac{\tau^*(p^{\nu_1+\cdots +
\nu_r)})}{p^{\nu_1 s_1+\cdots +\nu_r s_r}}
\end{equation*}
\begin{equation*}
= \prod_p \left(1+ \sum_{\substack{\nu_1,\ldots,\nu_r=0\\
\nu_1+\cdots +\nu_r\ge 1}}^{\infty} \frac{2}{p^{\nu_1 s_1+\cdots
+\nu_r s_r}}\right)
\end{equation*}
\begin{equation*}
= \prod_p \left(2 \left(1-\frac1{p^{s_1}}\right)^{-1}\cdots
\left(1-\frac1{p^{s_r}}\right)^{-1} -1\right)
\end{equation*}
\begin{equation*}
=\zeta^2(s_1)\cdots \zeta^2(s_r) H_r(s_1,\ldots,s_r),
\end{equation*}
where in $H_r(s_1,\ldots,s_r)$ the coefficients of $1/p^{s_j}$ are zero ($1\le j\le r$).
\end{proof}

\begin{proof}[Proof of Theorem {\rm \ref{Th_asympt_tau_1k}}]
We prove formula \eqref{asympt_tau_1k}. According to Proposition \ref{Prop_Dir_ser_tau_prod},
\begin{equation} \label{conv_id_tau_1k}
\tau_{1,k}(n_1\cdots n_r) = \sum_{d_1m_1=n_1,\ldots, d_rm_r=n_r} f_{r,k}(d_1,\ldots, d_r)\tau_{1,k}(m_1)\cdots \tau_{1,k}(m_r)
\end{equation}
for every $n_1,\ldots,n_r\in \N$, where $f_{r,k}$ is a multiplicative function and symmetric in the variables. Therefore,
\begin{equation*}
S_{r,k}(x):= \sum_{n_1,\ldots, n_r\le x} \tau_{1,k}(n_1\cdots n_r) = \sum_{d_1,\ldots, d_r\le x} f_{r,k}(d_1,\ldots, d_r)
\prod_{j=1}^r \sum_{m_j\le x/d_j} \tau_{1,k}(m_j).
\end{equation*}

For $k\ge 2$ we deduce by \eqref{tau_1dim} that
\begin{equation*}
S_{r,k}(x)
\end{equation*}
\begin{equation} \label{factors}
= \sum_{d_1,\ldots, d_r\le x} f_{r,k}(d_1,\ldots, d_r) \prod_{j=1}^r \left(\zeta(k) \frac{x}{d_j} + \zeta(1/k)
\left(\frac{x}{d_j}\right)^{1/k} + O\left(\left(\frac{x}{d_j}\right)^{\theta_k+\varepsilon}\right)\right).
\end{equation}

Here the main term will be
\begin{equation*}
M_{r,k}(x):= (\zeta(k))^r x^r \sum_{d_1,\ldots, d_r\le x} \frac{f_{r,k}(d_1,\ldots, d_r)}{d_1\cdots d_r}
\end{equation*}
\begin{equation*}
=(\zeta(k))^r x^r \sum_{d_1,\ldots, d_r=1}^{\infty} \frac{f_{r,k}(d_1,\ldots, d_r)}{d_1\cdots d_r} + R_{r,k}(x)=
(\zeta(k))^r x^r F_{r,k}(1,\ldots,1) + R_{r,k}(x),
\end{equation*}
where $F_{r,k}(1,\ldots,1)$ is convergent and its value is by \eqref{conv_id_tau_1k},
\begin{equation*}
F_{r,k}(1,\ldots,1)= \prod_p \left(1-\frac1{p} \right)^r \left(1-\frac1{p^k} \right)^r \sum_{\nu_1,\ldots,\nu_r=0}^{\infty}
\frac{\tau_{1,k}(p^{\nu_1 +\cdots +\nu_r})}{p^{\nu_1+\cdots+ \nu_r}}
\end{equation*}
\begin{equation*}
=(\zeta(k))^{-r} \prod_p \left(1-\frac1{p} \right)^r \sum_{\nu_1,\ldots,\nu_r=0}^{\infty}
\frac{\lfloor (\nu_1 +\cdots +\nu_r)/k \rfloor +1}{p^{\nu_1+\cdots+ \nu_r}},
\end{equation*}
while
\begin{equation*}
R_{r,k}(x) \ll x^r  \sideset{}{'}\sum_{} \frac{|f_{r,k}(d_1,\ldots, d_r)|}{d_1\cdots d_r},
\end{equation*}
$\sum^{'}$ meaning that $d_1,\ldots,d_r\le x$ does not hold. That is, there is at least one $t$ such that $d_t>x$. We can assume,
without restricting the generality, that $t=1$. We obtain that
\begin{equation*}
\sideset{}{'}\sum_{d_1>x} \frac{|f_{r,k}(d_1,\ldots, d_r)|}{d_1\cdots d_r}= \sideset{}{'}\sum_{d_1>x}
\frac{|f_{r,k}(d_1,\ldots, d_r)|}{d_1^{\varepsilon}d_2\cdots d_r}
\cdot \frac1{d_1^{1-\varepsilon}}
\end{equation*}
\begin{equation*}
\ll \frac1{x^{1-\varepsilon}} \sum_{d_1,\ldots,d_r=1}^{\infty}
\frac{|f_{r,k}(d_1,\ldots, d_r)|}{d_1^{\varepsilon}d_2\cdots d_r}\ll \frac1{x^{1-\varepsilon}},
\end{equation*}
since the latter series is $F_{r,k}(\varepsilon,1,\ldots,1)$, which converges by Proposition \ref{Prop_Dir_ser_tau_prod}. This gives
$R_{r,k}(x) \ll x^{r-1+\varepsilon}$ and
\begin{equation} \label{main_term}
M_{r,k}(x)=A_{k,r} x^r + O(x^{r-1+\varepsilon}).
\end{equation}

By multiplying in \eqref{factors} the terms $\zeta(k)\frac{x}{d_j}$ ($1\le j\le r-1$) and $\zeta(1/k)(\frac{x}{d_r})^{1/k}$ we have
\begin{equation*}
(\zeta(k))^{r-1}\zeta(1/k) x^{r-1+1/k} \sum_{d_1,\ldots, d_r\le x} \frac{f_{r,k}(d_1,\ldots, d_r)}{d_1\cdots d_{r-1}d_r^{1/k}}
\end{equation*}
\begin{equation*}
=(\zeta(k))^{r-1}\zeta(1/k) x^{r-1+1/k} \sum_{d_1,\ldots, d_r=1}^{\infty} \frac{f_{r,k}(d_1,\ldots, d_r)}{d_1\cdots d_{r-1}d_r^{1/k}} + T_{r,k}(x)
\end{equation*}
\begin{equation*}
= (\zeta(k))^{r-1} \zeta(1/k) x^{r-1+1/k} F_{r,k}(1,\ldots,1,1/k) + T_{r,k}(x),
\end{equation*}
where $F_{r,k}(1,\ldots,1,1/k)$ is convergent and
\begin{equation*}
T_{r,k}(x) \ll x^{r-1+1/k}  \sideset{}{'}\sum_{} \frac{|f_{r,k}(d_1,\ldots, d_r)|}{d_1\cdots d_{r-1}d_r^{1/k}}
\end{equation*}
with $\sum^{'}$ as above. There are two cases. Case I: Assuming that $d_r>x$ we deduce
\begin{equation*}
\sideset{}{'}\sum_{d_r>x} \frac{|f_{r,k}(d_1,\ldots, d_r)|}{d_1\cdots d_{r-1}d_r^{1/k}} = \sideset{}{'}\sum_{d_r>x}
\frac{|f_{r,k}(d_1,\ldots, d_r)|}{d_1 \cdots d_{r-1} d_r^{\varepsilon}}
\cdot \frac1{d_r^{1/k-\varepsilon}}
\end{equation*}
\begin{equation*}
\ll \frac1{x^{1/k -\varepsilon}} \sum_{d_1,\ldots,d_r=1}^{\infty}
\frac{|f_{r,k}(d_1,\ldots, d_r)|}{d_1\cdots d_{r-1}d_r^{\varepsilon}} \ll \frac1{x^{1/k-\varepsilon}}.
\end{equation*}

Case II: If $d_r\le x$, then there is a $t\in \{1,\ldots,r-1\}$ such that $d_t>x$. We deduce by taking $t=1$,
\begin{equation*}
\sideset{}{'}\sum_{\substack{d_1>x\\ d_r\le x}} \frac{|f_{r,k}(d_1,\ldots, d_r)|}{d_1d_2\cdots d_{r-1}d_r^{1/k}} = \sideset{}{'}
\sum_{\substack{d_1>x\\ d_r\le x}} \frac{|f_{r,k}(d_1,\ldots, d_r)|}{d_1^{\varepsilon}d_2\cdots d_{r-1}d_r}
\cdot \frac{d_r^{1-1/k}}{d_1^{1-\varepsilon}}
\end{equation*}
\begin{equation*}
\ll \frac{x^{1-1/k}}{x^{1-\varepsilon}} \sum_{d_1,\ldots,d_r=1}^{\infty}
\frac{|f_{r,k}(d_1,\ldots, d_r)|}{d_1^{\varepsilon}d_2\cdots d_{r-1} d_r}\ll \frac1{x^{1/k-\varepsilon}}.
\end{equation*}

Hence
\begin{equation*}
T_{r,k}(x) \ll x^{r-1+\varepsilon},
\end{equation*}
the same error as in \eqref{main_term}.

All the terms obtained by multiplying in \eqref{factors} $\zeta(k)\frac{x}{d_j}$ ($j\in \{1,\ldots,r\}\setminus \{t\}$)
and $\zeta(1/k)(\frac{x}{d_t})^{1/k}$ are of the same size and give together
\begin{equation} \label{second_term}
B_{r,k} x^{r-1+1/k} + O(x^{r-1+\varepsilon}),
\end{equation}
where
\begin{equation*}
B_{r,k}= r (\zeta(k))^{r-1} \zeta(1/k)F_{r,k}(1,\ldots,1,1/k).
\end{equation*}

Now, if in \eqref{factors} we take an error term, say $O((\frac{x}{d_r})^{\theta_{k}+\varepsilon})$, then we have to consider
$\zeta(k)\frac{x}{d_j}$ ($1\le j\le r-1$) to obtain, by multiplying, the largest term, which is
\begin{equation*}
\ll x^{r-1+\theta_k+\varepsilon} \sum_{d_1,\ldots,d_r\le x}
\frac{|f_{r,k}(d_1,\ldots, d_r)|}{d_1\cdots d_{r-1} d_r^{\theta_k+\varepsilon}}
\end{equation*}
\begin{equation*}
\ll x^{r-1+\theta_k+\varepsilon} \sum_{d_1,\ldots,d_r=1}^{\infty}
\frac{|f_{r,k}(d_1,\ldots, d_r)|}{d_1\cdots d_{r-1} d_r^{\theta_k+\varepsilon}},
\end{equation*}
giving the error
\begin{equation} \label{final_error}
x^{r-1+\theta_k+\varepsilon},
\end{equation}
since the involved series is convergent.

Therefore, \eqref{asympt_tau_1k} follows by \eqref{main_term}, \eqref{second_term} and \eqref{final_error}.

The proof of \eqref{asympt_tau_ik_lcm} is by similar arguments, based on Proposition \ref{Prop_Dir_ser_tau_lcm}.
\end{proof}

\begin{proof}[Proof of Theorem {\rm \ref{Th_asympt_tau_exp}}]
Similar to the proof of Theorem \ref{Th_asympt_tau_1k} by selecting $k=2$, using Proposition \ref{Prop_Dir_ser_tau_exp_prod_lcm} and the
fact that the behavior of $\tau^{(e)}$ is similar to $\tau_{1,2}(n)$, as explained before.
\end{proof}

\begin{proof}[Proof of Theorem {\rm \ref{Th_gen_convo}}]
For each $1\leq j\leq r,$ we write by condition (i),
$$F_j(x)=M_j(x)+E_j(x),$$
where
$$ M_j(x)=x^{a_j}P_j(\log x),\ \  E_j(x)= O(x^{b_j}).$$

Then we have
\begin{align} \label{sum_h}
\sum_{n_1,\ldots, n_r\le x} h(n_1\cdots n_r)& = \sum_{d_1,\ldots, d_r\le x} g(d_1,\ldots, d_r)
\prod_{j=1}^r F_j(x/d_j)\\
&
=  \sum_{d_1,\ldots, d_r\le x} g(d_1,\ldots, d_r) \prod_{j=1}^r \left(M_j(x/d_j)+E_j(x/d_j) \right).\nonumber
\end{align}

It is easy to see that we can write
\begin{align} \label{easy}
&\prod_{j=1}^r \left(M_j(x/d_j)+E_j(x/d_j) \right)=\prod_{j=1}^r  M_j(x/d_j)+\eta(x;d_1,\ldots,d_r),\\
&\eta(x;d_1,\ldots,d_r)\ll \sum_{j=1}^r\left(\frac{x}{d_j}\right)^{b_j} \prod_{\substack{1\leq k\leq r\\ k\not= j}}
\left(\frac{x}{d_k}\right)^{a_k}\times (\log x)^{\delta_1+\cdots+\delta_r}.\nonumber
\end{align}

Let $$L_j(x):=x^{a_1+\cdots+a_{j-1}+b_j+a_{j+1}+\cdots +a_r} \quad (1\leq j\leq r).$$

The contribution of $\eta(x;d_1,\ldots,d_r)$ is
\begin{align} \label{contrib}
&\ll \sum_{d_1,\ldots, d_r\le x} |g(d_1,\ldots, d_r)|\times |\eta(x;d_1,\cdots,d_r)|\\
&\ll (\log x)^{\delta_1+\cdots+\delta_r} \sum_{j=1}^r L_j(x) \sum_{d_1,\ldots, d_r\le x}
\frac{|g(d_1,\ldots, d_r)|}{d_1^{a_1}\cdots d_{j-1}^{a_{j-1}}d_j^{b_j}d_{j+1}^{a_{j+1}}\cdots d_r^{a_r}}\nonumber\\
& \ll  x^{a_1+\cdots+a_r-\Delta}(\log x)^{\delta_1+\cdots+\delta_r},\nonumber
\end{align}
where we used condition (ii).

Now we evaluate the sum
$$ M(x):= \sum_{d_1,\ldots, d_r\le x} g(d_1,\ldots, d_r)
\prod_{j=1}^r M_j\left(\frac{x}{d_j}\right).$$

Since $M_j(u)=x^{a_j}P_j(u)$ with $P_j(u)$ a polynomial in $u$ of degree $\delta_j,$ we have
\begin{align*}
\prod_{j=1}^r M_j\left(\frac{x}{d_j}\right) = \frac{x^{a_1+\cdots +a_r}}{d_1^{a_1}\cdots d_r^{a_r}}
\sum_{\ell=0}^{\delta_1+\cdots +\delta_r}C_{\ell}(\log d_1,\ldots, \log d_r) (\log x)^{\ell},
\end{align*}
where
$$ C_{\ell}(\log d_1,\ldots, \log d_r)=\sum_{j_1,\ldots,j_r} c(j_1,\ldots,j_r) (\log d_1)^{j_1} \cdots (\log d_r)^{j_r},
$$
the sum being over $0\leq j_t\leq \delta_t$ ($1\leq t\leq r$). So we have
\begin{align} \label{M_x}
M(x)&= x^{a_1+\cdots +a_r}
\sum_{\ell=0}^{\delta_1+\cdots +\delta_r} (\log x)^{\ell}
\sum_{d_1,\ldots, d_r\le x} \frac{g(d_1,\ldots, d_r) C_{\ell}(\log d_1,\ldots, \log d_r)}{d_1^{a_1}\cdots d_r^{a_r}} \\
 & = x^{a_1+\cdots +a_r}
\sum_{\ell=0}^{\delta_1+\cdots +\delta_r} d_{\ell} (\log x)^{\ell}
\nonumber\\
 &\ \ +\ \ x^{a_1+\cdots +a_r} \sum_{\ell=0}^{\delta_1+\cdots +\delta_r} (\log x)^{\ell}
\sideset{}{'} \sum_{d_1,\ldots, d_r} \frac{g(d_1,\ldots, d_r) C_{\ell}(\log d_1,\ldots, \log d_r)}{d_1^{a_1}\cdots d_r^{a_r}}, \nonumber
\end{align}
where
$$d_{\ell}:= \sum_{d_1 ,\ldots, d_r=1 }^\infty \frac{g(d_1,\ldots, d_r) C_{\ell}(\log d_1,\ldots, \log d_r)}{d_1^{a_1}\cdots d_r^{a_r}}
$$
and where $\sum^{'}$  means that there is at least one $j$ ($1\leq j\leq r$) such that $d_j>x.$ Without loss of generality, we suppose $d_r>x.$

Suppose $\varepsilon>0$ is sufficiently small and we have $\log n\ll n^{\varepsilon}.$ Thus
\begin{align} \label{M_x_estimate_cont}
&x^{a_1+\cdots +a_r}
\sum_{\ell=0}^{\delta_1+\cdots +\delta_r} (\log x)^{\ell}
\sideset{}{'} \sum_{d_1,\ldots, d_r} \frac{g(d_1,\ldots, d_r) C_{\ell}(\log d_1,\ldots, \log d_r)}{d_1^{a_1}\cdots d_r^{a_r}}\\
&\ll x^{a_1+\cdots +a_r}
\sum_{\ell=0}^{\delta_1+\cdots +\delta_r} (\log x)^{\ell} \sum_{d_1,\ldots, d_{r-1}= 1}^{\infty} \sum_{d_r>x} \frac{|g(d_1,\ldots, d_r)|
 d_1^{\delta_1 \varepsilon}\cdots d_r^{\delta_r \varepsilon}}{d_1^{a_1}\cdots d_r^{a_r}}\nonumber\\
 &\ll x^{a_1+\cdots +a_r}
\sum_{\ell=0}^{\delta_1+\cdots +\delta_r} (\log x)^{\ell} \sum_{d_1,\ldots, d_{r-1}= 1 }^{\infty} \sum_{d_r>x} \frac{|g(d_1,\ldots, d_r)|
 d_1^{\delta_1 \varepsilon} \cdots d_r^{\delta_r \varepsilon}}{d_1^{a_1}\cdots d_{r-1}^{a_{r-1}} d_r^{b_r}} \times\frac{1}{d_r^{a_r-b_r}}
 \nonumber\\
 &\ll x^{a_1+\cdots +a_r-(a_r-b_r)}
\sum_{\ell=0}^{\delta_1+\cdots +\delta_r} (\log x)^{\ell} \sum_{d_1 ,\ldots, d_{r}= 1 }^{\infty}  \frac{ |g(d_1,\ldots, d_r)|}{d_1^{a_1-\delta_1 \varepsilon}\cdots d_{r-1}^{ a_{r-1}-\delta_{r-1}\varepsilon} d_r^{ b_r-\delta_r\varepsilon}  }\nonumber\\
  &\ll x^{a_1+\cdots +a_r-(a_r-b_r)} (\log x)^{\delta_1+\cdots +\delta_r} \ll x^{a_1+\cdots +a_r- \Delta} (\log x)^{\delta_1+\cdots +\delta_r},
  \nonumber
\end{align}
by using condition (ii) again.

From \eqref{M_x} and \eqref{M_x_estimate_cont} we get
\begin{equation} \label{we_get}
M(x)=x^{a_1+\cdots +a_r}
\sum_{\ell=0}^{\delta_1+\cdots +\delta_r} d_{\ell} (\log x)^{\ell} + O(x^{a_1+\cdots +a_r- \Delta} (\log x)^{\delta_1+\cdots + \delta_r}).
\end{equation}

Now Theorem \ref{Th_gen_convo} follows from \eqref{sum_h}, \eqref{easy}, \eqref{contrib} and \eqref{we_get}.
\end{proof}

\begin{proof}[Proof of Theorem {\rm \ref{Th_tau_r}}] Apply Theorem \ref{Th_gen_convo} in the case $f_j(n)=\tau(n)$, $a_j=1$,
$b_j=\theta+\varepsilon$, $\delta_j=1$ ($1\le j\le r$) by using Proposition \ref{Prop_Dir_ser_tau_prod} and Proposition
\ref{Prop_Dir_ser_tau_lcm} for $k=1$.
\end{proof}

\begin{proof}[Proof of Theorem {\rm \ref{Th_tau_star_r}}] Apply Theorem \ref{Th_gen_convo} in the case $f_j(n)=\tau^*(n)$, $a_j=1$,
$b_j=1/2+\varepsilon$, $\delta_j=1$ ($1\le j\le r$) by using Proposition \ref{Prop_Dir_ser_tau_star_prod}.
\end{proof}

\begin{proof}[Proof of Theorem {\rm \ref{Th_tau_1_k_gcd_r}}] Let $f$ be an arbitrary arithmetic function. Then
\begin{equation*}
\sum_{n_1,\ldots,n_r\le x} f((n_1,\ldots,n_r)) = \sum_{n_1,\ldots,n_r\le x} \sum_{d\mid (n_1,\ldots,n_r)} (\mu*f)(d)
\end{equation*}
\begin{equation*}
= \sum_{dj_1,\ldots, dj_r\le x} (\mu*f)(d) = \sum_{d\le x}(\mu*f)(d)  \lfloor x/d \rfloor^r.
\end{equation*}

In the case $f=\tau_{1,k}$ we have $(\mu*\tau_{1,k})(n)=1$ ($n=m^k$), $0$ (otherwise), and deduce
\begin{equation*}
\sum_{n_1,\ldots,n_r\le x} \tau_{1,k}((n_1,\ldots,n_r)) = \sum_{m^k\le x} \left( \frac{x}{m^k}+O(1)\right)^r
\end{equation*}
\begin{equation*}
= \sum_{m\le x^{1/k}} \left(\frac{x^r}{m^{kr}}+O\left(\frac{x^{r-1}}{m^{k(r-1)}}\right)\right)
\end{equation*}
\begin{equation*}
= x^r \sum_{m=1}^{\infty} \frac1{m^{kr}}+ O\left(x^r \sum_{m>x^{1/k}} \frac1{m^{kr}}\right)+
O\left(x^{r-1} \sum_{m\le x^{1/k}} \frac1{m^{k(r-1)}}\right),
\end{equation*}
and \eqref{tau_1_k_gcd_r} is obtained by standard elementary estimates.
\end{proof}

\section{Acknowledgements}
The authors thank the referee for useful comments and suggestions, which improved the paper.

This work was supported by the National Key Basic Research Program of
China (Grant No. 2013CB834201). Part of this paper was completed while the first author visited the China University of Mining and Technology
in November 2016.

\medskip

\noindent L\'aszl\'o T\'oth \\
Department of Mathematics \\
University of P\'ecs \\
Ifj\'us\'ag \'utja 6, H-7624 P\'ecs, Hungary
\\ E-mail: {\tt ltoth@gamma.ttk.pte.hu}

\medskip \medskip

\noindent Wenguang Zhai \\
Department of Mathematics \\
China University of Mining and Technology \\
Beijing 100083, China \\
E-mail: {\tt zhaiwg@hotmail.com}

\end{document}